\documentclass[12pt,reqno]{amsart}

\usepackage[utf8]{inputenc}
\usepackage{xcolor}

\usepackage{xcolor}

\theoremstyle{plain}
\newtheorem{theorem}{Theorem}
\newtheorem*{remark}{Remark}
\newtheorem{proposition}{Proposition}[section]

\newtheorem{example}{Example}

\numberwithin{equation}{section}

\DeclareMathOperator{\arcsinh}{\mathrm{arcsinh}}
\DeclareMathOperator{\arctanh}{\mathrm{arctanh}}

\begin{document}

	\title[Solutions to the harmonic map problem between pseudo-Riemannian surfaces]{Solutions to the harmonic map problem between pseudo-Riemannian surfaces
	}
	
	\author{A. Fotiadis}
	
	\email{A. Fotiadis:  fotiadisanestis@math.auth.gr}
	\author{C. Daskaloyannis}
	\email{C. Daskaloyannis: daskalo@math.auth.gr }
	\curraddr{
		Department of Mathematics, Aristotle University of Thessaloniki,
		Thessaloniki 54124, Greece}
	\date{}
	

	\begin{abstract}
		We study locally harmonic maps between a Riemann surface or Lorentz surface $M$ and a Riemann surface or Lorentz surface $N$. {All four cases are studied in a unified way}. All four cases are written using a unified formalism. Therefore solutions to the harmonic map problem can be studied in a unified way.
		
		Harmonic maps between pseudo-Riemannian surfaces are classified by the classification of the solutions of a generalized sine-Gordon equation. We then study the one-soliton solutions of this equation and we find the corresponding harmonic maps in a unified way.
		
		Next, we discuss a B\"acklund transformation of the harmonic map equations that provides  a connection between the solutions of two sine-Gordon type equations. Finally, we give an example of a harmonic map that is constructed by the use of a B\"acklund transformation.
	\end{abstract}
	
	\maketitle

	\section{Introduction and Statement of the Results}\label{Introduction}

	The theory of harmonic maps between surfaces has been studied extensively. Harmonic maps in the case when the domain and/or the target is a {Lorentz} surface have been studied much, because of the importance of {Lorentz} spaces in mathematical physics. The construction of harmonic maps from a Riemann or Lorentz surface to a Riemann or Lorentz surface has many applications. We {call} pseudo-Riemannian a surface that is either a Riemann or Lorentz surface.

	The aim of this article is to develop a method to construct harmonic diffeomorphisms between a Riemann surface or a Lorentz surface $M$ and a Riemann surface or a Lorentz surface $N$. The case when $N$ is of constant curvature is studied in more detail. Our results cover all cases of negative and positive constant curvature in a unified way, stemming from the study of a generalized sine-Gordon equation (that is either an elliptic or hyperbolic sine-Gordon or sinh-Gordon equation).

	To state our results we first need to introduce some notation. Let $u=(R,S): {M\rightarrow N}$, where $M$ (resp. $N$) is a Riemann or Lorentz surface with a metric $g$ (resp. $h$). Consider an isothermal coordinate system $(x,y)$ on $M$ such that
	\begin{equation*}
		g=e^{f(v,w)} dv dw,
	\end{equation*} 
	where $v=x+\epsilon y, w=x-\epsilon y, \,\epsilon=1, i.$ Thus, \[g=e^{f(x,y)}(dx^2-\epsilon^2 dy^2).\] The case $\epsilon=1$ (resp. $\epsilon=i$) corresponds to the case of Lorentz (resp. Riemann) surface $M$. 
	Similarly, consider an isothermal coordinate system $(R,S)$ on $N$ such that
	\begin{equation*}
		h =e^{F(V,W)}  dV dW,
	\end{equation*}
	where $V=R+\delta S, W=R-\delta S, \,\delta=1, i.$ Thus, \[h=e^{F(R,S)}(dR^2-\delta^2 dS^2).\] The case $\delta=1$ (resp. $\delta=i$) corresponds to the case of Lorentz (resp. Riemann) surface $N$. 
	
	We say that $u=(R,S)$ is \emph{a harmonic map }if it satisfies  the Euler-Lagrange equation for the energy functional 
	
	\begin{align*}
	E(u)&=\frac{1}{2}\int_{M} \, g^{ij}(p)h_{a b}(u(p))u^{a}_{i}(p)u^{b}_{j}(p)dp \\
	&=\int_{M}  \, e^{F(V,W)}(V_v W_{w}+V_{w} W_v) \, dv dw,
\end{align*}
	where 
		\begin{align*}\label{eq:Vv Vw}
		V_{v}&=\frac{1}{2}(V_{x}+\frac{1}{\epsilon}V_{y})=\frac{1}{2}(R_{x}+\frac{1}{\epsilon}R_{y}+\delta S_{x}+\frac{\delta}{\epsilon}S_{y}),\\
		 V_{w}&=\frac{1}{2}(V_{x}-\frac{1}{\epsilon}V_{y})=\frac{1}{2}(R_{x}-\frac{1}{\epsilon}R_{y}+\delta S_{x}-\frac{\delta}{\epsilon}S_{y}),\\
		 W_{v}&=\frac{1}{2}(W_{x}+\frac{1}{\epsilon}W_{y})=\frac{1}{2}(R_{x}+\frac{1}{\epsilon}R_{y}-\delta S_{x}-\frac{\delta}{\epsilon}S_{y}),\\
		  W_{w}&=\frac{1}{2}(W_{x}-\frac{1}{\epsilon}W_{y})=\frac{1}{2}(R_{x}-\frac{1}{\epsilon}R_{y}-\delta S_{x}+\frac{\delta}{\epsilon}S_{y}).
		\end{align*}

	Note that 
	\[
	{V_v W_{w}+V_{w} W_v} =\frac{1}{2}(R_{x}^2-\frac{1}{\epsilon^2}R_{y}^2-\delta^2 S_{x}^2+\frac{\delta^2}{\epsilon^2}S_{y}^2),
	\]
	thus in the case $\epsilon=\delta=i$ we get above what is known as the energy functional for a map between {Riemann} surfaces in isothermal coordinates.
	
	The \emph{generalized sine-Gordon} that we are interested in, is the equation 
	\[
	\frac{\delta}{\epsilon}\omega_{vw} =-\dfrac{ {K}_{N}}{2}\sinh{(2\frac{\delta}{\epsilon}\omega)},{ \qquad \delta, \,\epsilon \in \{ 1,i\},}
	\]
	The generalized sine-Gordon equation in Cartesian coordinates is as follows:

	\[\frac{\delta}{\epsilon}(\omega_{xx}-\frac{1}{\epsilon^2}\omega_{yy})=-2K_{N} \sinh(2\frac{\delta}{\epsilon}\omega).\]

	\emph{A harmonic map is said to be nontrivial if $V_v {W}_v  V_w  {W}_w \neq 0$. We shall assume that all maps are nontrivial}. When $\epsilon=\delta=i$ then $ V_{w}\equiv 0$ if $u$ is holomorphic and  $V_{v} \equiv 0$ if $u$ is antiholomorphic. Throughout this article, we also assume that the Jacobian of the harmonic map does not vanish and that the coordinates are the \emph{specific} ones (see next section for more details). The main results in this article are {the following.}
	
	\begin{theorem}\label{th:direct}
		Let $u=(R,S): M \to N$ be a harmonic map and let $V=R+\delta S, \, W=R-\delta S.$ Then, $(V,W)$ satisfy the  Beltrami equations:  
		\begin{equation*}\label{eq:Beltrami_Equation_specific1}
			\frac{V_w}{V_v}=  e^{ -2 \Omega}\,  ,\quad  
			\frac{W_{v}}{W_{w}}= e^{ -2 \Omega}\,  ,
		\end{equation*}
		where 
		\[
		\Omega=\frac{\delta}{\epsilon}\omega
		\]
		for a real valued function $\omega=\omega(x,y)$.
		Furthermore, 
		\[K_{N }=-\dfrac{2 \Omega_{vw}}{\sinh{(2\Omega)}},
		\]
		where $K_{N }$ is the curvature of the target manifold $N $.
	\end{theorem}

	These equations are well known for the case of a harmonic map from a Riemann surface to a Riemann surface \cite{FotDask}. In the present work we have four distinct cases depending on whether the domain and/or the target is a Riemann or a Lorentz surface.

	\begin{theorem}\label{Prop:Beltrami_to_harmonic}
		Consider the diffeomorphism $u=(R,S): M \to N$. Let and $V=R+\delta S, \, W=R-\delta S, \,v=x+\epsilon y, \, w=x-\epsilon y, \, \epsilon, \,\delta
		\in \{1, i\}$. If $(V,W)$ satisfy the Beltrami equations
		\begin{equation*}\label{Beltrami2}
			\frac{V_w}{V_v}=  e^{ -2 \Omega}\,  ,\quad  
			\frac{W_{v}}{W_{w}}= e^{ -2 \Omega}\,  ,
		\end{equation*}
		where 
		\[
		\Omega=\frac{\delta}{\epsilon}\omega
		\]
		for a real valued function $\omega=\omega(x,y)$, then $N$ can be equipped with a conformal metric such that $u$ is a harmonic map and the curvature of the target manifold $N $ is
		\[
		\tilde{K}_{N}=-\dfrac{2 \Omega_{vw}}{\sinh{(2\Omega)}}.
		\]
	\end{theorem}

	As a consequence, there is a classification of  harmonic diffeomorphisms via the classification of the  solutions of the generalized sine-Gordon equation and our results extend the results in \cite{FotDask}, which are referred to harmonic maps between Riemann surfaces.

	In \cite{FotDask} some examples of harmonic maps between Riemann surfaces of constant curvature were given. These examples were already known in the literature {\cite{Li-Ta2, S-T-W, Wolf1, Wolf2},} but in \cite{FotDask} there is a unified point of view to all these examples. Using one-soliton solutions of the generalized sine-Gordon equation we construct examples of harmonic maps to constant curvature (Lorentz or Riemann) surfaces. Elliptic functions arise as in \cite{FotDask} in course of this construction. We study positive, negative and zero constant curvature surfaces in a unified way, and we find concise formulas that provide new examples of harmonic maps.

	Finally, a B{\"a}cklund transform arises, which provides a connection between the solutions of two sine-Gordon type equations. As an application, we construct a new harmonic map in Section \ref{sec:Backlund}, that does not correspond to the one-soliton solutions of the generalized sine-Gordon equation.

	Let us now present an outline of the article. In Section \ref{sec:Background}, we start with some prelimaries and some necessary formulas as well as the {proofs of Theorems \ref{th:direct} and \ref{Prop:Beltrami_to_harmonic}.} In Section \ref{sec:Constant}, we study the constant curvature case. We find the one-soliton solutions of the generalized sine-Gordon equation and the corresponding harmonic maps. In Section \ref{sec:Backlund} we discuss the B\"acklund transform and by using this transform we construct a new harmonic map. Finally, Section \ref{sec:prespectives} contains some perspectives for future research.

	Complete surfaces of constant negative curvature immersed in the Euclidean 3-space are known as Pseudospherical Surfaces and they have been  studied as solutions of the sine-Gordon equation. There is a correspondence between Pseudospherical Surfaces and Lorentzian Harmonic Maps. So, the problem of classifying these surfaces reduces to solving the harmonic map problem \cite{DB}. On the other hand, the study of surfaces with constant curvature
	in Minkowski space is being developed thanks to its relationship with some classical equations such as
	sinh-Gordon, cosh-Gordon, sin-Laplace, and its associated B{\"a}cklund and Darboux transformations.
	
	The cases of {a} Riemann or Lorentz image surface with constant curvature have a special importance.  In this case we end up solving a sine-Gordon type equation, which plays an important role in many branches of mathematics and theoretical physics (see for example \cite{sinGordon}).  The sinh-Gordon  and sine-Gordon equations and their elliptic analogues have many applications and they have both been the subject of extensive study (see for example \cite{F-P,H}). The sinh-Gordon equation was also crucial to the breakthrough work \cite{Wente} on the Wente torus.

	There is a Kenmotsu-type representation for timelike surfaces with prescribed Gauss
	map in the three-dimensional Minkowski space. The complete timelike surfaces with positive constant curvature are classified in terms of harmonic diffeomorphisms between simply connected Lorentz surfaces and the universal cover of the de Sitter Space \cite{AEG}. 
	
	Wave maps (or Lorentzian-harmonic maps) from a 1+1-dimensional Lorentz space into the 2-sphere are associated to constant negative curvature surfaces in Euclidean 3-space via the Gauss map, which is harmonic with respect to the metric induced by the second fundamental form. The Cauchy problem for the wave map equation has been studied extensively over the past three decades. Gu showed the global existence of smooth wave maps from 2-dimensional Minkowski spacetime into a complete Riemannian manifold \cite{Gu}. Note finally that it is known that there is a relation between the harmonic map equation and anti-de Sitter geometry, see for example \cite{BS}.
	
	\section{Preliminaries and Main Results}\label{sec:Background}	
	
	In this section we shall start with some preliminaries and some necessary formulas. More precisely, we shall define isothermal coordinates and harmonic maps. Next, we shall find a necessary and sufficient condition for a map to be harmonic and we shall define the specific coordinates. Finally, we shall prove our main results. As a consequence, we shall derive a new B\"acklund transform. 
	\subsection{Isothermal Coordinates}

	Let {$u: (M,g) \rightarrow (N,h)$} be a map between a Riemannian surface or Lorentzian surface $M$ and a Riemannian surface or Lorentzian surface $N$.  The map $u$ is locally represented by $u=(R,S)$.
	
	It is known that isothermal coordinates on an arbitrary  {Riemann or Lorentz surface with a real analytic metric}  exist (see \cite[Section 8, p. 396]{J} and \cite[p. 111]{Larsen}). Consider an isothermal coordinate system $(x,y)$ on $M$ such that
	\begin{equation*}
		g=e^{f(x,y)}(dx^2 - \epsilon^2 dy^2)=e^{f(v,w)} dv dw,
	\end{equation*} 
	where $v=x+\epsilon y,\, w=x-\epsilon y,  \,\epsilon
	\in \{1, i\}.$
	Similarly, consider an isothermal coordinate system $(R,S)$ on $N$ such that
	\begin{equation*}
		h =e^{F(R,S)}(dR^2 - \delta^2 dS^2)=e^{F(V,W)}  dV dW,
	\end{equation*}
	where  
	\[
	V=R+\delta S, { \quad W =R-\delta S, \quad \delta\in \{1, i\}.}
	\]
	
	The curvature on the target is given by \cite{I-ENEA-CRE}:
	\begin{equation*}\label{curvature}
		K_N
		=-2F_{V W} e^{-F}=-\dfrac{1}{2}\left(F_{RR}-{\delta^2}F_{SS} \right)\, e^{-F(R,S)}. 
	\end{equation*} 
	
	\subsection{Harmonic Maps and the specific coordinate system}  \label{section Harmonic Maps}
	
	In the case of isothermal coordinates (see \cite[Section 8, p. 397]{J} and \cite[p. 186]{Tataru}),  a map $u=(R,S)$ is harmonic if it satisfies  the Euler-Lagrange equation for the energy functional \[
	E(u)=\frac{1}{2}\int_{M}  \, e^{F(V,W)}(V_v W_{w}+V_{w} W_v) \, dv dw,
	\]
	where in these coordinates the metrics on $M$ and $N$ are $g=e^{f(v,w)} dv dw$ and $h=e^{F(V,W)} dV dW$ respectively.

	\begin{proposition}\label{prop: eqharmonic}
		A map $u=(R,S)$ is harmonic if it satisfies the system
		\begin{align}\label{eq:harmonic_map}
			V_{vw} + F_{V}(V,W)\, V_{v} {V}_{w}&=0 \\
			\label{eq:harmonic_map_1}
			W_{vw} + F_{W}(V,W)\, W_{v} W_{w}&=0,
		\end{align}
		{where $V=R+\delta S$, $W=R-\delta S$.}
	\end{proposition}
	
	\begin{proof}
		
		Consider the Lagrangian
		\[\mathcal{L}(v,w,V,W,V_v,V_{w},W_v,W_{w})=e^{F(V,W)}(V_v W_{w}+V_{w} W_v).\]
		The Euler-Lagrange equations for the energy functional are
		\[
		\frac{\partial\mathcal{L}}{\partial V}= \left(\frac{\partial \mathcal{L}}{\partial V_v} \right)_v+\left( \frac{\partial \mathcal{L}}{\partial V_{w}}\right)_{w}
		\]
		and
		\[
		\frac{\partial\mathcal{L}}{\partial W}= \left(\frac{\partial \mathcal{L}}{\partial W_v} \right)_v+\left( \frac{\partial \mathcal{L}}{\partial W_{w}}\right)_{w} .
		\]
		A simple calculation reveals that the Euler-Lagrange equations are the equations in (\ref{eq:harmonic_map}) and  (\ref{eq:harmonic_map_1}). 
	\end{proof}
	Thus {\eqref{eq:harmonic_map} and \eqref{eq:harmonic_map_1}} are the equations of a harmonic map, and the one equation is `conjugate' to the other. 
	Observe that the above proposition is valid for any kind of harmonic map, as a map from a {Riemann or Lorentz domain surface to some Riemann or Lorentz} target surface.

	From now on we shall only consider nontrivial harmonic maps with nonvanishing determinant.
	
	\begin{proposition}\label{prop:Hopf}
		A necessary and sufficient condition for  $u$ to be a harmonic map, is that 
		\begin{equation}\label{eq:Hopf}
			e^{F(V,W)}  V_v  {W}_v={\Lambda(v)}
		\end{equation}	
		and   
		\begin{equation}\label{eq:Hopf_1}	
			e^{F(V,W)}V_{w}   {W}_{w}={M({w})}
		\end{equation}
		where $\Lambda=\Lambda(v)$ is a smooth function of $v$, and $M=M(w)$ {is} its `conjugate'.
		
	\end{proposition}
	\begin{remark} Notice that the above equations are valid for any kind of harmonic map. The domain manifold $M$ could be a {Riemann or a Lorentz surface and also the image manifold could be Riemann or Lorentz.}  This is a unified description  of the harmonic maps between surfaces.
	\end{remark}
	For convenience, we shall {sketch} the proof. 
	\begin{proof}
		If $u$ is a harmonic map then one can apply equations (\ref{eq:harmonic_map}) {and (\ref{eq:harmonic_map_1})} and prove that 
		\[
		({e^{F(V,W)}  V_v  {W}_v})_{w}= 0.
		\]
		Indeed, one can deduce that 
		\[
		{ ({e^{F(V,W)}  V_v  {W}_v})_{w}= e^F \left( W_v (V_{vw} + F_{V}V_{v} V_{w})+V_v \left(W_{vw} + F_{W} W_{v} W_{w} \right) \right)=0. }
		\]
		The proof of the formula 
	{	\[
		{ \left(e^{F(V,W)} V_{w}   {W}_{w}\right)}_{ v}= 0.
		\]}
		is similar, thus omitted. 
		
		If on the other hand 
		\[
		e^{F(V,W)}  V_v  {W}_v= {\Lambda(v)}  
		\mbox{ and }  
		e^{F(V,W)}V_{w}   {W}_{w}=M({w}),
		\]
		then differentiating these equations in terms of {$w$ and $v$} respectively, we find 
		\[ W_v (V_{v w} + F_{V}V_{v} V_{w})+ V_{v}(W_{vw} + F_{W} V_{v} W_{w})=0,\]
		and
		\[ W_{w} (V_{v w} + F_{V}V_{v} V_{w})+V_{w} (W_{v w} + F_{W} W_{v} W_{w})=0,\]
		
		Taking into account that 
		\[V_{v} W_{w}-W_v V_{w}=\frac{\delta}{\epsilon}(R_{x}S_{y}-S_{x}R_{y})\neq 0,\]
		{we obtain}  (\ref{eq:harmonic_map}) and (\ref{eq:harmonic_map_1}).
	\end{proof}

	{
	In the domain surface $M$ of a harmonic map, one can choose a specific coordinate system in order to considerably facilitate the calculations. Note that $u$ is non-trivial, thus $\Lambda(v)\neq 0$ and $M(w)\neq 0$. This \textit{specific}  system is defined by the transformations
	\begin{equation}\label{eq:specific}
		\zeta= \int \dfrac{1}{\sqrt{\Lambda(v)}} \, dv,  \quad
		\eta= \int \dfrac{1}{\sqrt{M(w)}} \, dw 
	\end{equation}  
}

{	In this specific system the equations (\ref{eq:Hopf}),\; (\ref{eq:Hopf_1}) can be simplified by substituting $\Lambda(v)=1$ and $M(w)=1.$ From now on, we shall consider this specific coordinate system. Then, 
	\begin{equation}\label{eq:holomorphic_map_specific}
		{e^{F(V,W)}}  V_v   W_v= 1\quad \mbox{and}\quad	{e^{F(V,W)}}  V_w   W_w= 1
	\end{equation}
} 
{	Observe now that equations (\ref{eq:holomorphic_map_specific}) can be written as follows:
	\begin{align}\label{eq:Hopf_2}
	\dfrac{e^{F(V,W)}}{4} \left( R_{x}^2+\dfrac{1}{\epsilon^2}R_{y}^2-\delta^2 S_{x}^2-\dfrac{\delta^2}{\epsilon^2}S_{y}^2
		+\dfrac{2}{ \epsilon}\left(  R_x R_y-\delta^2 S_x S_y \right)  \right) =1
	\end{align}
}
and   
{
	\begin{align}\label{eq:Hopf_3}
		 \dfrac{e^{F(V,W)}}{4}	 \left( R_{x}^2+\dfrac{1}{\epsilon^2}R_{y}^2-\delta^2 S_{x}^2-\dfrac{\delta^2}{\epsilon^2}S_{y}^2-
		\dfrac{2}{ \epsilon}\left(  R_x R_y-\delta^2 S_x S_y  \right)\right)=1.
	\end{align}
}
	The above equations imply that
	\begin{equation}\label{eq:orthogonality} 
		R_x R_y \, -\delta^2 \, S_x S_y=0
	\end{equation}
	and
	\begin{equation}\label{eq:squares}
		\left( R_{x}^2+\frac{1}{\epsilon^2}R_{y}^2\right)  -\delta^2  \left( S_{x}^2+\frac{1}{\epsilon^2}S_{y}^2 \right) =4e^{-F(R,S)}.
	\end{equation}
	Note that (\ref{eq:orthogonality})  means that this specific system is an orthogonal coordinate system. Then (\ref{eq:orthogonality}) and (\ref{eq:squares}) imply that in this specific system the following relations hold true
	\begin{align}\label{eq:cart}
		R_x &= 2 e^{-\frac{F}{2}}\cosh{\Omega}\cosh{\Theta}   \\
		R_y &= 2\, \epsilon \,e^{-\frac{F}{2}}\sinh{\Omega}\sinh{\Theta} \\
		S_x &= \frac{2}{\delta} e^{-\frac{F}{2}}\cosh{\Omega}\sinh{\Theta} \\
		S_ y&= \frac{2\epsilon}{\delta} e^{-\frac{F}{2}}\sinh{\Omega}\cosh{\Theta}, \label{eq:cart4}
	\end{align}	
	\label{cartesian_dRdS}
	for some functions $\Omega, \Theta$.
	A case by case study reveals that since $R,S$ are real valued functions then 
	\begin{equation*}\label{eq:DefOmega}
		\Omega=\frac{\delta}{\epsilon}\omega,\quad \Theta=\delta \theta,
	\end{equation*}
	where $\omega=\omega(x,y),\;  \theta=\theta(x,y)$ are real valued functions.

	\subsection{Proof of the main results\label{Proof of the main results}}
	\begin{proof}[Proof of Theorem \ref{th:direct}]
		{First of all, when } we consider a specific coordinate system then Proposition \ref{prop:Hopf} imply equations (\ref{cartesian_dRdS}). Then the solutions of the harmonic map equations satisfy the following equations
		\begin{equation*}\label{eq:Vx}
			\begin{array}{ll}
				V_x=  2\, e^{-\frac{F}{2}}\cosh{\Omega} e^{\Theta} , & 
				W_x=  2\, e^{-\frac{F}{2}}\cosh{\Omega} e^{-\Theta}\\
				V_y=  2\, \epsilon\, e^{-\frac{F}{2}}\sinh{\Omega} e^{\Theta} , &  
				W_y= - 2\,  \epsilon\, e^{-\frac{F}{2}}\sinh{\Omega} e^{-\Theta}\\ 
			\end{array}
		\end{equation*}
		and thus
		\begin{equation}\label{eq:Vv}
			\begin{array}{ll}
				V_{v}= \frac{1}{2} (V_x +\dfrac{1}{\epsilon}V_y) = 2 \,e^{-\frac{F}{2}} e^{\Omega+\Theta} , & 
				W_{v}= \frac{1}{2} (W_x +\dfrac{1}{\epsilon}W_y) = 2 \, e^{-\frac{F}{2}} e^{-\Omega-\Theta}\\
				V_{w}=\frac{1}{2} (V_x -\dfrac{1}{\epsilon}V_y) =  2  \, e^{-\frac{F}{2}} e^{-\Omega+\Theta} , &  
				W_{w}=\frac{1}{2} (W_x -\dfrac{1}{\epsilon}W_y) =  2  \, e^{-\frac{F}{2}} e^{\Omega-\Theta}.\\ 
			\end{array}
		\end{equation}

		From equation (\ref{eq:Vv}) we have the following relations
		\begin{align}
			\Omega + \Theta&=   \dfrac{F}{2}+\log{\dfrac{V_{v}}{2} } ,
			\label{eq:pOmegapTheta}
			\\
			-\Omega + \Theta&=   \dfrac{F}{2}+\log{\dfrac{V_{w}}{2} }.
			\label{eq:mOmegapTheta}
		\end{align} 
		
		Taking into consideration equations (\ref{eq:holomorphic_map_specific}), it follows that
		equations (\ref{eq:pOmegapTheta})  and  (\ref{eq:mOmegapTheta})  imply the relations
		\begin{align*}
			\Omega_{w} + \Theta_{w}&=   \dfrac{F_{w}}{2}-F_{V} V_{w}  ,
			\\
			-\Omega_{v}+ \Theta_{v}&=   \dfrac{F_{v}}{2}-F_{V}  V_{v} .
		\end{align*} 
		A lengthy calculation reveals that the following equations hold true:
		\begin{equation*}\label{eq:second_der_Omega_Theta}
			\begin{split}
				\Omega_{{v}{w}} + \Theta_{{v}{w}}=   \dfrac{F_{{{v}{w}}}}{2}-F_{V}  V_{{{v}{w}}}  - F_{VV}  V_{w}  V_{v}  - F_{VW} V_{{w}}  W_{{v}} 
				\\
				-\Omega_{{v} {w}} + \Theta_{{v}{w}}=   \dfrac{F_{{v}{w}}}{2}-F_{V}  V_{{v}{w}}  - F_{VV}  V_{{v}}  V_{{w}}  - F_{VW} V_{{v}}  W_{{w}}.
			\end{split} 
		\end{equation*}
		We conclude that $\Omega$ satisfies the equation
		\begin{equation*}
			\Omega_{ {v}{w}}=\dfrac{1}{2} {F_{VW}}  \left(   V_{v} W_{w} -V_{w} W_{v}\right).
		\end{equation*}
		Taking in consideration (\ref{eq:Vv}), we deduce that $\Omega$ satisfies the sine-Gordon type equation
		\begin{equation*}\label{eq:sinh-Gordon_specific1}
			\Omega_{{v}{w}}= - \dfrac{K_N}{2}    \sinh 2 \Omega,
		\end{equation*} 
		or equivalently
		\begin{equation*}
			\Omega_{xx}-\epsilon^2 \Omega_{yy}= - 2{K_N } \,   \sinh 2 \Omega .
		\end{equation*} 
		Recall that $K_N$ is the curvature on the target 
		\begin{equation*}\label{curvature}
			K_N
			=-2F_{V W} e^{-F}=-\dfrac{1}{2}\left(F_{RR}-{\delta^2}F_{SS} \right)\, e^{-F(R,S)}. 
		\end{equation*} 
	\end{proof}

	We shall now prove Theorem \ref{Prop:Beltrami_to_harmonic}. Note that Theorem \ref{Prop:Beltrami_to_harmonic} is the converse of Theorem \ref{th:direct}, when the map is a diffeomorphism.

	\begin{proof}[Proof of Theorem \ref{Prop:Beltrami_to_harmonic}]
		Suppose now that the diffeomorphism $u$ is such that
		\begin{equation*}\label{Beltrami2}
			\frac{V_{w}}{V_{{v}}}= e^{ -2 \Omega},
		\end{equation*}
		and 
		\begin{equation*}\label{Beltrami2}
			\frac{W_{w}}{W_{{v}}}= e^{ 2 \Omega},
		\end{equation*}
		where $\Omega=\frac{\delta}{\epsilon}\omega$ for a real valued function $\omega$ {and $V=R+\delta S$, $W=R-\delta S$.}

		These formulas imply {by \eqref{eq:Vv} that }
		\[
		R_x R_y \, -\delta^2 \, S_x S_y=0.
		\]
		
		As a consequence, {by \eqref{eq:Vv} we get}
		\[
		V_{v}  W_{v}=\frac{1}{4}\left( R_{x}^2+\frac{1}{\epsilon^2}R_{y}^2-\delta^2 S_{x}^2-\frac{\delta^2}{\epsilon^2}S_{y}^2\right)\in \mathbb{R}.
		\]
		
		Then, we observe that
		\begin{equation*}
			\dfrac{1}{V_v  W_v}=\dfrac{1}{ V_{w}  {W}_{w}}= \dfrac{e^{2\Omega}}{ V_{v} {W}_{w}}\in \mathbb{R}.
		\end{equation*}

		Assume now, without loss of generality, that $\dfrac{e^{2\Omega}}{V_{v}{W_{w}}}>0.$
		We define
		\begin{equation*}
			\label{eq:tildeF}
			e^{\tilde{F}(V,W)}= \dfrac{e^{2\Omega}}{V_{v}W_{w}}.
		\end{equation*}
		Given the fact that the map $u$ is a diffeomorphism, we can consider $\tilde{F}$ as a function of $V$ and $W$.
		
		Let us equip $N$ with the conformal metric 
		\begin{equation*}\label{eq:metric}
			\tilde{h}=e^{\tilde{F}}dVdW.
		\end{equation*} 
		Then, 
		\begin{equation*}\label{3.11a}
			e^{\tilde{F}} V_v  W_v=1,
		\end{equation*}
		thus
		\begin{equation*}
			{V_{vw}} + \tilde{F}_{V}(v,w) V_{v} V_{w}=0.
		\end{equation*}
		Similarly, we can prove that 
		\begin{equation*}
			{W_{vw}} + \tilde{F}_{W}(v,w) W_{v} W_{w}=0.
		\end{equation*}
		Thus, {by Proposition \ref{prop: eqharmonic}, the maps $u=(R,S)$, where $V=V+\delta S$, $W=R-\delta S$, is harmonic.} Then, from Theorem \ref{th:direct} follows that the corresponding curvature {on the target} is 
		\begin{equation*}\label{star}
			\tilde{K}_{N}=
			-\frac{2\Omega_{vw}}{\sinh{2\Omega}}.
		\end{equation*} 
		Therefore Theorem \ref{Prop:Beltrami_to_harmonic} is valid. 
	\end{proof}

	{We finally discuss a B{\"a}cklund transformation arising from the analysis carried out in this section. The formulas (\ref{eq:cart})-(\ref{eq:cart4})} imply that the following system  of first order  PDEs holds true:
	
	\begin{equation*}\label{eq:system_1} 
		\begin{array}{c}
			\sinh \Omega\, R_x - \dfrac{\delta}{\epsilon}\cosh \Omega\, S_y =0 
			\\
			\dfrac{1}{\epsilon}\cosh\Omega\, R_y -\delta\sinh \Omega\, S_x=0  . 
		\end{array}
	\end{equation*}
	


	
	



	Using (\ref{eq:system_1})  we can compute mixed derivatives of $R$ and $S$. Then, the compatability condition $ (R_x)_y=(R_y)_x,\; (S_x)_y=(S_y)_x$ implies the following second order {PDEs} 
	\begin{align}
		\epsilon^2 \left(  \tanh\Omega\, R_x \right)_x {-}\left(  \coth\Omega\, R_y \right)_y &=0
		\label{eq:2_system_1}
		\\
		\epsilon^2 \left(  \tanh\Omega\, S_x \right)_x {-} \left(  \coth\Omega\, S_y \right)_y &=0.
		\label{eq:2_system_2}
	\end{align}
	Then, equation (\ref{eq:2_system_1}) can be written as
	\begin{equation}\label{eq:2nd_order_R}
		\tanh\Omega\,   R_{xx}-  \dfrac{1}{\epsilon^2}\coth\Omega\,   R_{yy}
		+\dfrac{ \Omega_x\,   R_x}{ \cosh^2\,\Omega}+ \dfrac{1}{\epsilon^2}\dfrac{ \Omega_y\,   R_y}{ \sinh^2\,\Omega}=0.
	\end{equation}
	This is a  second order PDE. Given a solution $\Omega=\frac{\delta}{\epsilon}\omega$ of the generalized sine-Gordon equation {\eqref{eq:sinh-Gordon_specific1}}, 
	then one can solve equation (\ref{eq:2nd_order_R}), by using standard methods, see for example \cite[p. 72]{P-R}. In the case of a Riemannian domain surface we have that $\epsilon =i$ and the second order PDE is an elliptic one, whereas in the case of a Lorentzian domain surface we have that $\epsilon=1$ and the second order PDE is a hyperbolic one. Similar results hold for $S=S(x,y)$.

	{Using {\eqref{eq:cart}-\eqref{eq:cart4} } we can compute mixed derivatives of $R$ and $S$. Then, the compatability condition $ (R_x)_y=(R_y)_x,\; (S_x)_y=(S_y)_x$ and a lengthy calculation implies the B\"{a}cklund transformation that is given by the following proposition.}

	\begin{proposition}
		\label{prop:Baecklund}	
		The functions $\Omega, \Theta$ are such that
		\begin{equation}\label{eq:Backlund}
			\begin{array}{rl}
				\Omega_x  -\frac{1 }{\epsilon }  \Theta_y
				=& \frac{1}{2}
				F_x \tanh\, \Omega
				,\\
				\Omega_y-\epsilon  \Theta_x
				=& \frac{1}{2} F_y
				\coth\, \Omega,
			\end{array}
		\end{equation}
		where
		\begin{align*}
			F_x&=  F_{R}  R_x  + F_{S} S_x \\
			F_y&=  F_{R}  R_y  + F_{S} S_y .
		\end{align*}
	\end{proposition}
	We will study in more detail this B\"{a}cklund transformation in Section \ref{sec:Backlund}. Several examples of non trivial harmonic maps in the case of Riemannian surfaces are given in \cite{FotDask}. In this article we generalize this construction to include the case when the target and/or the donain is a Lorentz surface.

	The strategy to find solutions of the harmonic map problem is to take a solution of the generalized sine-Gordon equation, find a solution of the Beltrami equations and find the metric on $N$ of constant curvature.

	
	\section{Constant curvature spaces, one-soliton solution}\label{sec:Constant}
	
	In this section we shall apply our main results in the case when $N$ is of constant curvature $K_{N}$. We construct a new family of harmonic maps using one-soliton solutions of the generalized sine-Gordon equation.

	In the specific coordinates (\ref{eq:specific}), the generalized sine-Gordon equation is 
	\[
	\Omega_{x x} +\frac{1}{\epsilon^2}\Omega_{y y}=-\frac{K_N}{2} \sinh 2 \Omega, \mbox{ where } K_N \mbox{ is constant}.
	\]
	The general solution for this equation is not known and there are only partial solutions. Of particular interest is the so called one-soliton solution, defined by:
	\[
	\Omega=\Omega( \gamma y-\beta x  ).
	\]
	We shall first find the one-soliton solutions of the generalized sine-Gordon equation and then we shall find the corresponding harmonic maps. Without loss of generality, we shall assume that $\gamma^2 -\epsilon^2 \beta^2=\rho^2>0$.  Set
	\[
	\gamma=\rho \cosh(\epsilon\tau)  \mbox{  and  }  \beta=\frac{\rho}{\epsilon} \sinh(\epsilon\tau).
	\]

	In order to simplify the calculations, we shall introduce a new system of coordinates $(X,Y)$:
	\begin{equation}\label{defn rho and tau}
		\begin{array}{l}
			X=  \gamma x - \epsilon^2 \beta y= x \rho \cosh(\epsilon\tau)-{y\epsilon  \rho} \sinh(\epsilon\tau)
			\\
			Y= -\beta x +\gamma y=- {x\dfrac{\rho}{\epsilon}} \sinh(\epsilon\tau)+ y \rho \cosh(\epsilon\tau).
		\end{array}
	\end{equation} 
	These coordinates are given by a contraction and an appropriate  rotation of the {Riemann (or the Lorentz)} domain surface.
	In these coordinates, $\Omega=\Omega(Y)$ and the generalized sine-Gordon equation is written
	\begin{equation*}\label{eq:mod_Sinh_G}
		\dfrac{  d^2 \Omega}{dY^2}=\dfrac{ 2K_N \epsilon^2}{\rho^2} \sinh 2\Omega  \; \mbox{  or  }  \;
		\left(\dfrac{  d \Omega}{dY}\right)^2=
		C
		+ \dfrac{ 4 K_N \epsilon^2} { \rho^2}
		\sinh^2 \,\Omega.
	\end{equation*}
	
	Case I: {$C\ne 0$.} 
	
	The above equation can be written as 
	{\begin{equation*}
			\label{eq:der_sinh_omega_equation}
			\left(\dfrac{d \sinh\,\Omega}{d(\sqrt{C}Y)} \right) ^2 = ( 1+ \sinh^2 \Omega)  \left(1+ \dfrac{ 4 K_N \epsilon^2} {C \rho^2}  
			\sinh^2 \Omega\right).
	\end{equation*}} 
	Set
	\begin{equation*}
		C=  (\dfrac{\delta}{\epsilon} c)^2.
	\end{equation*}
	The definitions and properties of the Jacobi elliptic functions are used in this paper. For more details, see for example \cite{DLMF} and \cite{MilThom64}.  The Jacobi elliptic function $sn$  is defined by the formula
\begin{equation}\label{eq:Jacobi_ampli}
\int_0^x \, \dfrac{dt}{   \sqrt{ \left( 1-t^2\right)\left( 1 -n t^2\right)}}=
sn^{-1}(x\vert n).
\end{equation}

	From \cite[Equation (22.13.9)]{DLMF}
	we have that
	\begin{equation}\label{eq:sinh_omega}
		\begin{array}{c}
			\sinh  \Omega  = 
			\mathrm{sc}(\dfrac{\delta}{\epsilon} c  \,
			(Y-Y_0 +\psi)|m) ,\\
			\mbox{  where } m= 1-\dfrac{ 4 K_N \epsilon^2} {C \rho^2}=
			1-\dfrac{ 4 K_N \delta^2} {c^2 \rho^2}.
		\end{array}
	\end{equation}
	
	For simplicity, we shall only consider the case that $c$ is a positive real number.

	Let 
	\begin{equation}\label{defn of phi}
		\phi= \dfrac{\delta}{\epsilon} c \, (Y-Y_0 +\psi).
	\end{equation} 
	The parameter $Y_0$ corresponds to the choice of the initial conditions. 
	From  \cite[Equations (22.6.2)  and (22.6.3)]{DLMF}, we conclude that
	\begin{align*}
		\cosh\, \Omega & =   \mathrm{nc}  (\phi|m) \\
		\tanh\, \Omega & =   \, \mathrm{sn}(\phi|m) \\
		\Omega'	&=\dfrac{\delta}{\epsilon}  \; \mathrm{dc}(\phi|m).
	\end{align*}
	Set
{	\begin{equation*}\label{constants}
			\Omega_0=\Omega(Y_0), \quad \Omega '_0=\Omega '(Y_0).
	\end{equation*}
}
	We are now ready to compute the harmonic maps {$u=(R,S)$} that correspond to the one-soliton solutions of the generalized sine-Gordon equation.
	
	Equation (\ref{defn rho and tau}) implies that:
	\begin{equation*}
		\begin{array}{rl}
			\dfrac{\partial}{\partial \, X}=
			&  \dfrac{\cosh(\epsilon \tau)}{\rho} \dfrac{\partial}{\partial \, x}\,+\, \dfrac{\sinh(\epsilon \tau)}{\epsilon \rho} \dfrac{\partial}{\partial \, y}\\
			\\
			\dfrac{\partial}{\partial \, Y}=&  \dfrac{\epsilon\sinh(\epsilon \tau)}{\rho} \dfrac{\partial}{\partial \, x}\,+\, \dfrac{\cosh(\epsilon \tau)}{ \rho} \dfrac{\partial}{\partial \, y}.
		\end{array}
	\end{equation*}
	Therefore {\eqref{eq:cart}-\eqref{eq:cart4} imply} that:
	{
		\begin{equation*}\label{eq:cartesian_XY}
			\begin{array}{rl}
				R_X=&\frac{1}{\rho} \,e^{-F/2}
				\left(
				e^{\Theta } \cosh (\Omega +\epsilon \tau)+e^{-\Theta } \cosh (\Omega -\epsilon \tau)
				\right)\\
				R_Y=&\frac{\epsilon}{\rho} \, e^{-F/2}
				\left(
				e^{\Theta } \sinh (\Omega +\epsilon \tau)-e^{-\Theta } \sinh (\Omega -\epsilon \tau)
				\right)\\ 
				S_X=&\frac{1}{\delta \rho}\,e^{-F/2}
				\left(
				e^{\Theta } \cosh (\Omega +\epsilon \tau) - e^{-\Theta } \cosh (\Omega -\epsilon \tau)
				\right)\\ 
				S_Y=&\frac{\epsilon}{\delta \rho} \,e^{-F/2}
				\left(
				e^{\Theta } \sinh (\Omega +\epsilon \tau)+ e^{-\Theta } \sinh (\Omega -\epsilon \tau)
				\right).
			\end{array}
		\end{equation*}
	}
	

	
	We are interested in a solution satisfying the constraints 
	{\[
		R_X=1  \quad \mbox{and} \quad S_X=0 .
		\]}
	In this case we have
	
	\begin{align*}	e^{\Theta } \cosh (\Omega +\epsilon \tau) - e^{-\Theta } \cosh (\Omega -\epsilon \tau)&=0 \\
		\frac{e^{-F/2}}{\rho}
		e^{\Theta } \cosh (\Omega +\epsilon \tau)+\frac{e^{-F/2}}{\rho}
		e^{-\Theta } \cosh (\Omega -\epsilon \tau)&=1
	\end{align*}	
	and the last two equations imply that
	\begin{equation}
		e^F=  \dfrac{4}{\rho^2}  \cosh\left( \Omega+\epsilon \tau \right) \cosh\left( \Omega-\epsilon \tau \right)= 
		\dfrac{2}{\rho^2}   \left( 
		\cosh\left(  2\, \Omega \right)
		+\cosh\left(  2\,\epsilon \tau \right) \right) .
		\label{eq:soliton_F}
	\end{equation}
	Furthermore,
	\begin{equation}\label{eq:RY_SY}
		\begin{array}{rl}
			R_Y &=\frac{\epsilon}{2}
			\left(   \tanh \left( \Omega+ \epsilon \tau \right) - \tanh \left( \Omega- \epsilon \tau \right)\right)= \epsilon \, \dfrac{  \sinh (2 \epsilon \tau)}{\cosh (2
				\epsilon \tau)+\cosh (2  \Omega)},
			\\
			S_Y &=\frac{\epsilon}{2 \delta}
			\left(   \tanh \left( \Omega+ \epsilon \tau \right) + \tanh \left( \Omega- \epsilon \tau \right)\right)=\frac{\epsilon}{ \delta}  \dfrac{  \sinh (2 \Omega)}{\cosh (2
				\epsilon \tau)+\cosh (2  \Omega)}.
		\end{array}
	\end{equation}
	
	From equation (\ref{eq:sinh_omega}) {it} follows that
	\begin{align*}
		\cosh(2\epsilon \tau)+\cosh(2 \Omega)&= 2 
		\left( \mathrm{sc}^2\left( \phi|m\right) +\cosh^2(\epsilon \tau)\right)\\
		&=
		\dfrac{2}{1-m}\left( \mathrm{dc}^2(\phi|m)-\mathcal{N}^2
		\right),
	\end{align*}
	where
	\[
	\mathcal{N}^2=1-(1-m)  \sinh^2(\epsilon\tau),
	\]
	and $\phi$ as in (\ref{defn of phi}).
	Then 
	\[
	\mathrm{e}^F= \dfrac{4}{\rho^2} \left( \mathrm{sc}^2\left( \phi|m\right) +\cosh^2(\epsilon \tau)\right)=
	\dfrac{4}{\rho^2 (1-m)}\left( \mathrm{dc}^2(\phi|m)-\mathcal{N}^2
	\right)
	>0.
	\]
	We shall now find $S$. From Equation (\ref{eq:RY_SY}) we deduce that
	\[
	S_Y= \dfrac{\epsilon}{\delta} 
	\dfrac{\cosh\Omega  \sinh\Omega}{  \sinh^2\Omega+  \cosh^2(\epsilon \tau)}=
	\dfrac{\epsilon}{\delta} \; 
	\dfrac{\mathrm{nc}(\phi|m)  \mathrm{sc}(\phi|m)}{  \mathrm{dc}^2(\phi|m) -  \mathcal{N}^2}                                                                                                                                                                                                                                                                                                                                                                                                                                                                                                                                                                                                                                                                                                                                                                                                                                                                                                                                                                                                                                                                                                                                                                                                                                                                                                                                                                                                                                                                                                                                                                                                                                                                                                                                                                                                                                                                                                                                                                                                                                                                                                                                                                                                                                                                                                                                                                                                                                                                                                                                                                                                                                                                                                                                                                                                                                                                                                                                                                                                                                                                                                                                                                                                                                                                                                                                                                                                                                                                                                                                                                                                                                                                                                                                                                                                                                                                                                                                                                                                                                                                                                                                                                     .                                                                                                                                                                                                                                                                                                                                                                                                                                                                                                                                                                                                                                                                                                                                                                                                                                                                                                                                                                                                                                                                                                                                                                                                                                                                                                                                                                                                                                                                                                                                                                 \]

	We conclude that
	\[
	S_Y= \dfrac{\epsilon}{\delta} \;  
	\dfrac{\mathrm{dc}'(\phi|m)}{\mathrm{dc}^2(\phi|m) -  \mathcal{N}^2},
	\]
	
	hence integrating
	\begin{equation*} \label{eq:S}
		S-S_0= - \left. \dfrac{1}{\left( \dfrac{\delta}{\epsilon}\right) ^2 c\mathcal{N}}
		\mathrm{arctanh}
		\left(
		\dfrac{\mathrm{dc}(\phi|m)}{\mathcal{N}}
		\right)\right| ^{\phi= \frac{\delta}{\epsilon} c \, (Y-Y_0 +\psi)}_{\phi= \frac{\delta}{\epsilon} c \, \psi}.
	\end{equation*}

	We shall now describe the metric $e^{F(R,S)}$. From (\ref{eq:soliton_F})   we find that
	
	\begin{equation*}
		\label{eq:F_soliton_S}
		e^F= {\dfrac{4  \mathcal{N}^2}{\rho^2 {(m-1)}}  \dfrac{1}{\cosh^2\left( \left( \dfrac{\delta}{\epsilon}\right) ^2{c\,  \mathcal{N}}(S-\Sigma_0)  \right)}},
	\end{equation*} 
	where
	\[
	\Sigma_0 = S_0 +  
	\dfrac{1}{\left( \dfrac{\delta}{\epsilon}\right) ^2\,c\,  \mathcal{N}}\arctanh \dfrac{\mathrm{dc}(\dfrac{\delta}{\epsilon} c \, \psi |m)}{\mathcal{N}}.
	\]

	We shall now find $R$. From (\ref{eq:RY_SY}), we deduce that
	\[
	R_Y
	=-\epsilon  \dfrac{\sinh (2 \epsilon \tau)}{\mathcal N ^2}
	\dfrac{1}{1-\dfrac{1}{\mathcal N ^2}  \mathrm{dc}^2(\dfrac{\delta}{\epsilon} c \, (Y-Y_0 +\psi) |m)} {{\frac{(1-m)}{2}}}.
	\]
	Integrating and using the identity
	\[
	dc(z|m)=sn(z+K|m) , \mbox{ where }\; K=K(m)= \int_{0}^{\frac{\pi}{2}} 
	\dfrac{d\theta}{\sqrt{ 1 -{m} \sin^2
			\theta}},
	\]

	we find that
{	\begin{align*}\label{eq:R}
			R-R_0&=X-X_0-\dfrac{\epsilon^2}{\delta}\dfrac{\sinh (2 \epsilon \tau)}{ c\, \mathcal N ^2}{{\frac{(1-m)}{2}}} \times \notag\\
			&\times\Bigg(
			\Pi\left( \dfrac{1}{\mathcal N ^2};\dfrac{\delta}{\epsilon} c\,(Y-Y_0+\psi) +K  |m \right) 
			-
			\Pi\left( \dfrac{1}{\mathcal N ^2}; \dfrac{\delta}{\epsilon} c\,\psi +K  |m \right)
			\Bigg) ,
	\end{align*}
}
	where we used the definition formula of the third kind elliptic integral,
	\[
	\Pi\left( n;x|m\right) = \int_{0}^{x} \, \dfrac{du}{1-n\, \mathrm{sn}^2(u|m)}.
	\]

	Note that the metric (Riemannian or Lorentzian) on the target manifold $N$ is of constant curvature and that the results in Section \ref{sec:Constant} cover all cases of positive, negative and zero constant curvature in a unified formulation. 
	
	Case II: $C=0$. 
	
	In this case the generalized sine-Gordon equation is written as
	\[
	\left(\dfrac{  d \Omega}{dY}\right)^2=
	\dfrac{ 4 K_N \epsilon^2} { \rho^2}
	\sinh^2 \,\Omega,
	\]
	from which we deduce that $\frac{K_{N}}{\epsilon^2}>0.$ 
	
	A case by case study reveals that if $\frac{\delta^2}{\epsilon^2}=1$ then $\tanh{\omega}=\frac{1}{\cosh{(\frac{2}{\rho}\sqrt{\frac{K_{N}}{\epsilon^2}}Y)}}$ and if $\frac{\delta^2}{\epsilon^2}=-1$ then $\tan{\omega}=\frac{1}{\sinh{(\frac{2}{\rho}\sqrt{\frac{K_{N}}{\epsilon^2}}Y)}}$.
	
	Moreover, if $\frac{\delta^2}{\epsilon^2}=1$ then \[\tan{(\frac{2}{\rho}\sqrt{\frac{K_{N}}{\epsilon^2}}\cosh(\epsilon\tau )S)}=\cosh(\epsilon\tau ) \sinh{(\frac{2}{\rho}\sqrt{\frac{K_{N}}{\epsilon^2}}Y)}\]
	and
	\[
	e^{F}=\frac{4\cosh^2(\epsilon\tau )}{\rho^2 \sin^2{(\frac{2}{\rho}\sqrt{\frac{K_{N}}{\epsilon^2}}\cosh(\epsilon\tau )S)}}.
	\]
	
	If $\frac{\delta^2}{\epsilon^2}=-1$ then \[\coth{(\frac{2}{\rho}\sqrt{\frac{K_{N}}{\epsilon^2}}\cosh(\epsilon\tau )S)}=\cosh(\epsilon\tau ) \cosh{(\frac{2}{\rho}\sqrt{\frac{K_{N}}{\epsilon^2}}Y)}\]
	and
	\[
	e^{F}=\frac{4\cosh^2(\epsilon\tau )}{\rho^2 \cosh^2{(\frac{2}{\rho}\sqrt{\frac{K_{N}}{\epsilon^2}}\cosh(\epsilon\tau )S)}}.
	\]

	\section{B\"acklund Transform of the generalized sine-Gordon equation}\label{sec:Backlund}
	
	In this section we discuss a B\"acklund transform that provides a connection between the solutions of two sine-Gordon type equations. We provide a new harmonic map by using the B\"acklund transform of a solution to the generalized sine-Gordon equation. 
	
	The \textit{B\"acklund transformation} is a system of first order partial differential equations relating the solution of a PDE   (in our case the generalized sine-Gordon PDE) to the solution of  another PDE (in our case the  PDE (\ref{eq:theta_zeta_barzeta})). Then, the one solution is said to be the \textit{B\"acklund transform} of the other. 
	
	A direct consequence of Proposition \ref{prop:Baecklund} is the following proposition. 
	\begin{proposition}
		The  equation  
		\begin{equation}\label{eq:theta_zeta_barzeta}
			2  \Theta_{vw}  =
			e^{-F}
			\left(     \left( F_V^2-F_{VV}\right) e^{2  \Theta}   -  \left( F_{W}^2-F_{WW}\right) e^{-2  \Theta} \right)
		\end{equation}
		is the B\"acklund transform of the equation
		\begin{equation*}\label{eq:sinh-G_zeta}
			\Omega_{vw}=-\frac{
				K_N}{2} \sinh ({2\Omega}),
		\end{equation*}
		where $\Theta$ and $\Omega$ are related by (\ref{eq:Backlund}).
	\end{proposition}

	One interesting remark arising directly from the above, is the following.
	\begin{remark}
		Consider the case of the metric
		\[
		e^F=\frac{1}{(aR+\delta^2 b S)^2}.
		\]
		Then, $F_{V}=-(a+\delta b) e^{\frac{F}{2}}, F_{W}=-(a-\delta b)  e^{\frac{F}{2}}$. Thus, $K_N=-a^2+\delta^2 b^2$ and the function $\Theta$ satisfies the generalized sine-Gordon type equation
		\begin{equation*}\label{generalizedGordon}
			\Theta_{xx}-\epsilon^2 \Theta_{yy}={{(a+\delta b) }^2 e^{2\Theta }-{{(a-\delta b) }}^2 e^{-2\Theta }}.
		\end{equation*}

		The B{\"a}cklund transform can be written as follows
		\begin{equation*}
			\begin{array}{rl}
				\Omega_x  -\frac{1 }{\epsilon }  \Theta_y
				=& -2\sinh \Omega \left(a \cosh \Theta + \delta b \sinh \Theta \right)
				,\\
				\Omega_y-\epsilon  \Theta_x
				=& -2
				\cosh \Omega \left(a \sinh \Theta + \delta b \cosh \Theta \right).
			\end{array}
		\end{equation*}
		For example, if $\delta=\epsilon=1, a=1, b=0$ then we get an auto-B{\"a}cklund transform of the equation 
		\[
		\omega_{xx}- \omega_{yy}=  2 \sinh (2 \omega),
		\]
		and if $\delta=i, \epsilon=1, a=1, b=0$ then we get an auto-B{\"a}cklund transform of the equation 
		\[
		\omega_{xx}- \omega_{yy}=  2 \sin (2 \omega).
		\]
	\end{remark}
	
	As an application, let us give the following example.
	
	\begin{example}
		Consider
		\[
		\Theta(x,y)=\arcsinh{\left(\tanh{2\delta x}\right)},
		\] 
		that is a one-soliton solution  of the equation 
		\begin{equation*}
			\Theta_{xx}-\epsilon^2 \Theta_{yy}={2\delta^2 \sinh{2\Theta}}.
		\end{equation*}
		Then, 
		\[
		\Omega(x,y)=2 \arctanh{\frac{\cosh{(2\delta x)}}{2\delta \epsilon y}}
		\]
		is the B{\"a}cklund transform  of $\Theta$ and satisfies the equation
		\[
		\Omega_{xx}-\epsilon^2 \Omega_{yy}={-2\delta^2 \sinh{2\Omega}}.
		\]
		
		Then, a harmonic map $u=R+iS$ associated to the solution $\Omega$ is $u=(R,S),$ where
		\begin{equation*}
			\begin{array}{rl}
				R(x,y)
				=& \delta \epsilon^2 y^2 \tan(2\delta x)+\frac{x}{2}
				,\\
				S(x,y)
				=& y^2\frac{\epsilon^2}{\cos{2\delta x}}-\frac{\cos{2\delta x}}{2\delta^2},
			\end{array}
		\end{equation*}
		where $\epsilon,\, \delta \in \{1, i\}.$
		
		Finally, observe that the metric on the target $N$ is given by \[h=\frac{1}{S^2}\left( dR^2-\delta^2 dS^2\right).\]
		The curvature of this metric is $K_{N}=\delta^2$. 
		
		The metric on the domain $M$ is of the form \[ g=e^{f(x,y)}\left( dx^2-\epsilon^2 dy^2\right).\]
	\end{example}

\section{Perspectives for future investigation}\label{sec:prespectives}

Firstly, an application of the methods introduced in this paper is to find new solutions to the generalized sine-Gordon equation and the corresponding harmonic maps. The generalized sine-Gordon equation is interesting on its own right and we plan to construct new families of solutions. The application of the ideas exposed in this paper in multidimensional cases is under current investigation.

\end{document}